\documentclass{article}

\usepackage{fullpage}
\usepackage{amsmath,amsfonts,amssymb,amsthm,enumerate}
\usepackage[hidelinks]{hyperref}
\usepackage{cleveref,thmtools,thm-restate,color}

\usepackage[most]{tcolorbox}
\usepackage{xcolor}
\usepackage{orcidlink}

\newtcolorbox{aiusagebox}{
  colback=blue!3,
  colframe=blue!70!black,
  boxrule=1.4pt,
  arc=3mm,
  left=3mm,
  right=3mm,
  top=2mm,
  bottom=2mm
}

\newtheorem{theorem}{Theorem}[section]
\newtheorem{lemma}[theorem]{Lemma}
\newtheorem{corollary}[theorem]{Corollary}
\theoremstyle{definition}
\newtheorem{definition}[theorem]{Definition}
\theoremstyle{remark}

\DeclareMathOperator{\Sym}{Sym}
\DeclareMathOperator{\id}{id}
\DeclareMathOperator{\im}{im}

\providecommand{\oldbinary}{Theorem 1.3}
\providecommand{\oldunary}{Theorem 1.4}
\providecommand{\improvedunary}{Theorem 1.13}

\providecommand{\restfirst}[1]{|_{{#1}\ast}}
\providecommand{\restsecond}[1]{|_{\ast{#1}}}

\title{Triviality of promise polymorphisms}
\author{Yuval Filmus \orcidlink{0000-0002-1739-0872} \\
\small Taub Faculty of Computer Science and Faculty of Mathematics \\
\small Technion --- Israel Institute of Technology, Haifa, Israel \\
\small \texttt{yuvalfi@cs.technion.ac.il}}

\begin{document}

\maketitle

\begin{abstract}
Given two $m$-ary predicates $P,Q$, an $n$-ary polymorphism is a tuple $(f_1,\dots,f_m)$ of functions such that $x^{(1)},\dots,x^{(n)} \in P$ implies $(f_1(y_1),\dots,f_m(y_m)) \in Q$, where $y_i = (x^{(1)}_i,\dots,x^{(m)}_i)$. This generalizes the usual definition in universal algebra, in which $P = Q$ and $f_1 = \cdots = f_m$.

In earlier work, we studied when all polymorphisms of a single predicate are ``trivial'': either all depend on a single coordinate (common to all of them), or they constitute a ``certificate'' for the predicate. We showed that it suffices to check this condition for $2$-ary polymorphisms, and even for $1$-ary polymorphisms, modulo an explicit list of obstructions.

In this paper we generalize the first result to the $P,Q$ setting, for a relaxed notion of certificate. We also generalize the second result in the promise setting, in which $P,Q$ range over the same alphabets and $P \subseteq Q$.
\end{abstract}

\section{Introduction} \label{sec:introduction}

Which predicates support a nontrivial polymorphism? In the simplest setting, a predicate is a subset $P \subseteq \Sigma^m$, a polymorphism $f\colon \Sigma^n \to \Sigma$ is a function satisfying the condition
\[
 x^{(1)},\dots,x^{(n)} \in P \to (f(x^{(1)}_1,\dots,x^{(n)}_1), \dots, f(x^{(1)}_m,\dots,x^{(n)}_m)) \in P,
\]
and a polymorphism is trivial if it is constant or depends on a single coordinate. Post~\cite{Post20,Post42} classified all possible polymorphisms of predicates over $\{0,1\}$, and his classification answers the question completely when $\Sigma = \{0,1\}$.

In previous work~\cite{Filmus26}, we studied this question in the multi-sorted setting. In this setting, a predicate is a subset $P \subseteq \Sigma_1 \times \cdots \times \Sigma_m$, and a polymorphism is an $m$-tuple of functions $(f_1,\dots,f_m)$, where $f_i\colon \Sigma_i^n \to \Sigma_i$, satisfying the condition
\[
 x^{(1)},\dots,x^{(n)} \in P \to (f_1(x^{(1)}_1,\dots,x^{(n)}_1), \dots, f_m(x^{(1)}_m,\dots,x^{(n)}_m)) \in P.
\]

The work~\cite{Filmus26} proposes the following notion of triviality (which we present in a simplified form): a polymorphism $(f_1,\dots,f_m)$ is trivial if either of the following cases holds:
\begin{description}
    \item[Dictator type] There exists $s \in [n]$ such that $f_i(x)$ depends only on $x_s$ for all $i$.
    \item[Certificate type] There exist a subset $I \subseteq [m]$ and symbols $\sigma_i \in \Sigma_i$ for all $i \in I$ such that (i) $f_i \equiv \sigma_i$ for all $i \in I$, (ii) $x \in P$ whenever $x_i = \sigma_i$ for all $i \in I$.
\end{description}
In words, either $(f_1,\dots,f_m)$ depends on a single ``row'' (viewing the input as an $n \times m$ matrix), or the constant $f_i$ force the output to belong to~$P$. This notion of triviality is inspired by the work of Mossel~\cite{Mossel09,Mossel12} on Arrow's theorem without the assumption of unanimity.

Let us say that a predicate $P$ is \emph{certificate-trivial for $n$} if all $n$-ary polymorphisms of $P$ are either of dictator type or of certificate type. The paper~\cite{Filmus26} has two main results. The first result~\cite[\oldbinary]{Filmus26} states that $P$ is certificate-trivial for all $n$ iff it is certificate-trivial for $n=2$. The second main result~\cite[\oldunary]{Filmus26} improves this to $n=1$, with an explicit list of obstructions.

The follow-up work~\cite{Filmus26b} whittles down the list of obstructions under a modified definition of triviality, in which we require $f_1 = \cdots = f_m$ in the case of dictator type (this only makes sense when $\Sigma_1 = \cdots = \Sigma_m$)~\cite[\improvedunary]{Filmus26b}.

\medskip

In this paper, we generalize the results of~\cite{Filmus26,Filmus26b} in two directions: we consider a more relaxed notion of triviality, and we allow the ``input'' predicate (the one satisfied by $x^{(1)},\dots,x^{(n)}$) to differ from the ``output'' predicate (the one satisfied by the image of the polymorphism). Our results hold even when the alphabets are infinite.

\begin{definition}[Relation pair]
Let $P \subseteq \Sigma_1 \times \cdots \times \Sigma_m$ and let $Q \subseteq \Delta_1 \times \cdots \times \Delta_m$. An $n$-ary polymorphism of $(P,Q)$ is an $m$-tuple of functions $(f_1,\dots,f_m)$, where $f_i\colon \Sigma_i^n \to \Delta_i$, such that
\[
 x^{(1)},\dots,x^{(n)} \in P \to (f_1(x^{(1)}_1,\dots,x^{(n)}_1), \dots, f_m(x^{(1)}_m,\dots,x^{(n)}_m)) \in Q.
\]
\end{definition}

\begin{definition}[Box-triviality]
Let $\Phi \subseteq \Delta_1^{\Sigma_1} \times \cdots \times \Delta_m^{\Sigma_m}$. 

An $m$-tuple of functions $(f_1,\dots,f_m)$, where $f_i\colon \Sigma_i^n \to \Delta_i$, is $\Phi$-box-trivial if either of the following cases holds:
\begin{description}
    \item[Dictator type] There exist $s \in [n]$ and $(\phi_1,\dots,\phi_m) \in \Phi$ such that $f_i(x) = \phi_i(x_s)$ for all $i$.
    \item[Box type] $\im f_1 \times \cdots \times \im f_m \subseteq Q$.
\end{description}

A relation pair $(P,Q)$, where $P \subseteq \Sigma_1 \times \cdots \times \Sigma_m$ and $Q \subseteq \Delta_1 \times \cdots \times \Delta_m$, is $\Phi$-box-trivial for $n$, where $n \ge 1$, if every $n$-ary polymorphism of $(P,Q)$ is $\Phi$-box-trivial.
\end{definition}

The definition of dictator type is the same as in~\cite{Filmus26,Filmus26b}; the difference is that certificate type is replaced by box type, and that we allow $P$ to be different from $Q$.

Box type is more general than certificate type, as demonstrated by
\[
 P = Q = \{\text{all permutations of } (0,0,0), (2,2,2), (0,1,1), (0,1,2)\},
\]
which is $\{(\id,\id,\id)\}$-box-trivial for all $n$ but not $\{(\id,\id,\id)\}$-certificate-trivial even for $n = 1$.

Our first main result generalizes \cite[\oldbinary]{Filmus26}.

\begin{restatable}{theorem}{mainbinary} \label{thm:main-binary}
Let $P \subseteq \Sigma_1 \times \cdots \times \Sigma_m$, $Q \subseteq \Delta_1 \times \cdots \times \Delta_m$, and $\Phi \subseteq \Delta_1^{\Sigma_1} \times \cdots \times \Delta_m^{\Sigma_m}$.

Suppose that $|\Sigma_i| \ge 2$ for all $i$, and that $P$ has full projections: $P|_i = \Sigma_i$.

The relation pair $(P,Q)$ is $\Phi$-box-trivial for all $n$ iff it is $\Phi$-box-trivial for $n=2$.
\end{restatable}

The proof is a straightforward generalization of the proof of \cite[\oldbinary]{Filmus26}. Moreover, we can derive \cite[\oldbinary]{Filmus26} (in the setting of relation pairs) from \Cref{thm:main-binary}, as we show in \Cref{sec:certificate-version}.

Our second main result generalizes \cite[\improvedunary]{Filmus26b} to the setting of promise polymorphisms.

\begin{restatable}{theorem}{mainunary} \label{thm:main-unary}
Let $P \subseteq Q \subsetneq \Sigma_1 \times \cdots \times \Sigma_m$, and let $\Phi \subseteq \Sym(\Sigma_1) \times \cdots \times \Sym(\Sigma_m)$, where $\Sym(\Sigma)$ consists of all permutations of $\Sigma$.

Suppose that $|\Sigma_i| \ge 2$ for all $i$, that $P$ has full projections, and that $\Phi$ is \emph{synchronous}: if $\phi,\psi \in \Phi$ and $\phi_i = \psi_i$ then $\phi = \psi$.

The relation pair $(P,Q)$ is $\Phi$-box-trivial for all $n$ iff it is $\Phi$-box-trivial for $n=1$ and neither of the following hold:
\begin{description}
\item[Unit witness obstruction] There exist $i \in [m]$ and $b \in \Sigma_i$ such that $|\Sigma_i| = 2$ and $y \in Q$ whenever $y_i = b$.
\item[Generalized upset obstruction] $|\Sigma_i| = 2$ for all $i$, and there exists $y \in P$ such that for all $x \in Q \setminus \{y\}$, the predicate $Q$ contains all vectors obtained from $x$ by changing coordinates to their value in the complement of $y$.
\end{description}
In particular, if all alphabets have size larger than~$2$, then there are no obstructions.
\end{restatable}

When $\Sigma_1 = \cdots = \Sigma_m$, one way to satisfy the requirement of synchronicity is by positing that $\Phi$ consists of tuples of the form $(\phi,\dots,\phi)$.

Let us demystify the generalized upset obstruction. When $\Sigma_i = \{0,1\}$ for all $i$ and $y = (0,\dots,0) \in P$, it states that $Q = \{(0,\dots,0\}) \cup Q'$, where $Q'$ is monotone: if $x \in Q'$ and $x_i \leq z_i$ for all $i$ then $z \in Q'$.

\paragraph{Remark} The standard setting of promise polymorphisms is more general: the predicates $P,Q$ can range over different alphabets, and we only require the existence of a homomorphism from $P$ to $Q$. In our case it is natural to ask that the homomorphism be injective, in which case we can assume that $\Sigma_i \subseteq \Delta_i$ and $P \subseteq Q$. In this setting, the condition on $(\phi_1,\dots,\phi_m) \in \Phi$ should be that $\phi_1,\dots,\phi_m$ are injective.

\Cref{thm:main-unary} fails in this setting, as the following example demonstrates:
\begin{align*}
 P &= \{x \in \{0,1\}^3 \mid x \neq (0,0,0),(1,1,1)\} \subseteq \{0,1\}^3, \\
 Q &= P \cup \{x \in \{0,2\}^3 \mid x \neq (0,0,0)\} \subseteq \{0,1,2\}^3, \\
 \Phi &= \{(\phi,\phi,\phi) \mid \phi\colon \{0,1\} \to \{0,1,2\} \text{ is injective and } 0 \in \im\phi\}.
\end{align*}
One checks that $(P,Q)$ is $\Phi$-box-trivial for $n = 1$. In contrast, let
\[
 f(a,b) =
 \begin{cases}
     0 & \text{if } a = b = 1, \\
     2 & \text{otherwise}.
 \end{cases}
\]
Then $(f,f,f)$ is a $2$-ary polymorphism which is not $\Phi$-box-trivial.

\paragraph{Paper organization} We prove two simple preliminary results in \Cref{sec:prel}. We then prove \Cref{thm:main-binary} in \Cref{sec:reduction-to-binary} and \Cref{thm:main-unary} in \Cref{sec:reduction-to-unary}. We formulate and prove a certificate version of \Cref{thm:main-binary} in \Cref{sec:certificate-version}.
A Lean formalization of the results is available at
\href{https://github.com/YuvalFilmus/box-triviality}
{\texttt{YuvalFilmus/box-triviality}}.

\paragraph{Acknowledgments.} The work was supported by the Israel Science Foundation, Grant No.\ 507/24.

\begin{aiusagebox}
\begin{center}
\textbf{AI usage statement}
\end{center}

\vspace{0.3em}

ChatGPT 5.6 Sol came up with the definition of box-triviality, proved the main results, and formalized them in Lean.
The author improved the formulation of the main theorems, verified the correctness of the proofs, and takes full responsibility
for the final content.
\end{aiusagebox}

\section{Preliminary results}
\label{sec:prel}

In this section we prove two simple results:
\begin{itemize}
\item Box-triviality is ``closed downwards'': if $(P,Q)$ is $\Phi$-box-trivial for $n$, then it is $\Phi$-box-trivial for all $n' \leq n$.
\item In the definition of $\Phi$-box-triviality, we can assume that all $(\phi_1,\dots,\phi_m) \in \Phi$ are polymorphisms.
\end{itemize}

\begin{lemma} \label{lem:downwards}
Let $P \subseteq \Sigma_1 \times \cdots \times \Sigma_m$, $Q \subseteq \Delta_1 \times \cdots \times \Delta_m$, and $\Phi \subseteq \Delta_1^{\Sigma_1} \times \cdots \times \Delta_m^{\Sigma_m}$.

If the relation pair $(P,Q)$ is $\Phi$-box-trivial for $n$, then it is $\Phi$-box-trivial for all $n' \le n$.
\end{lemma}
\begin{proof}
Let $(f_1,\dots,f_m)$ be an $n'$-ary polymorphism of $(P,Q)$, for some $n' \leq n$. We construct an $n$-ary polymorphism $(g_1,\dots,g_m)$ of $(P,Q)$ by ignoring the last $n-n'$ coordinates:
\[
 g_i(x_1,\dots,x_n) = f_i(x_1,\dots,x_{n'}).
\]

By assumption, $(g_1,\dots,g_m)$ is either of dictator type or of box type. We consider three cases:
\begin{itemize}
\item $(g_1,\dots,g_m)$ is of dictator type for $s \leq n'$ and $(\phi_1,\dots,\phi_m) \in \Phi$.

In this case, $f_i(x) = \phi_i(x_s)$, and so $(f_1,\dots,f_m)$ is also of dictator type for the same $s$ and $(\phi_1,\dots,\phi_m)$.

\item $(g_1,\dots,g_m)$ is of dictator type for $s > n'$ and $(\phi_1,\dots,\phi_m) \in \Phi$.

Since $g_1,\dots,g_m$ do not depend on the $s$th coordinate, the functions $\phi_1,\dots,\phi_m$ must be constant. Consequently, $f_i(x) = \phi_i(x_1)$ for all $i$, and so $(f_1,\dots,f_m)$ is again of dictator type.

\item $(g_1,\dots,g_m)$ is of box type.

Since $\im f_i = \im g_i$ for all $i$, also $(f_1,\dots,f_m)$ is of box type. \qedhere
\end{itemize}
\end{proof}

\begin{lemma} \label{lem:Phi}
Let $P \subseteq \Sigma_1 \times \cdots \times \Sigma_m$, $Q \subseteq \Delta_1 \times \cdots \times \Delta_m$, and $\Phi \subseteq \Delta_1^{\Sigma_1} \times \cdots \times \Delta_m^{\Sigma_m}$.

Let $\Psi$ consist of those $(\phi_1,\dots,\phi_m) \in \Phi$ which are polymorphisms.

The relation pair $(P,Q)$ is $\Phi$-box-trivial for $n$ iff it is $\Psi$-box-trivial for $n$.
\end{lemma}
\begin{proof}
If $(P,Q)$ is $\Psi$-box-trivial for $n$ then it is also $\Phi$-box-trivial for $n$, since $\Psi \subseteq \Phi$.

In the other direction, suppose that $(P,Q)$ is $\Phi$-box-trivial for $n$, and let $(f_1,\dots,f_m)$ be an $n$-ary polymorphism of $(P,Q)$. By assumption, it is either of dictator type or of box type. If $(f_1,\dots,f_m)$ is of box type then we are done, so suppose that $(f_1,\dots,f_m)$ is of dictator type: for some $s$ and some $(\phi_1,\dots,\phi_m) \in \Phi$, we have $f_i(x) = \phi_i(x_s)$ for all $i$. To complete the proof, we show that $(\phi_1,\dots,\phi_m)$ is a polymorphism, and so belongs to $\Psi$.

Let $(\sigma_1,\dots,\sigma_m) \in P$. Since $(f_1,\dots,f_m)$ is a polymorphism, $(f_1(\sigma_1,\dots,\sigma_1),\dots,f_m(\sigma_m,\dots,\sigma_m)) \in Q$. Since $f_i(\sigma_i,\dots,\sigma_i) = \phi_i(\sigma_i)$, we see that $(\phi_1(\sigma_1),\dots,\phi_m(\sigma_m)) \in Q$, and so $(\phi_1,\dots,\phi_m)$ is a polymorphism.
\end{proof}

\section{Reduction to binary}
\label{sec:reduction-to-binary}

In this section we prove \Cref{thm:main-binary}, closely following the proof of the corresponding result \cite[\oldbinary]{Filmus26}.

\mainbinary*

In view of \Cref{lem:Phi}, we can assume without loss of generality that all members of $\Phi$ are polymorphisms.

Given a relation pair $(P,Q)$ which is $\Phi$-box-trivial for $n = 2$, we need to show that it is $\Phi$-box-trivial for all $n$. \Cref{lem:downwards} shows that $(P,Q)$ is $\Phi$-box-trivial for $n = 1$. We show that $(P,Q)$ is $\Phi$-box-trivial for $n \ge 3$ using an inductive argument, encapsulated in the following lemma.

\begin{lemma} \label{lem:main-binary}
Suppose that $P,Q,\Phi$ satisfy the assumptions of \Cref{thm:main-binary}.

If $(P,Q)$ is $\Phi$-box-trivial for $n = 2$ and for some $n \ge 1$, then $(P,Q)$ is also $\Phi$-box-trivial for $n + 1$.
\end{lemma}
\begin{proof}
Let $(f_1,\dots,f_m)$ be an $(n+1)$-ary polymorphism of $(P,Q)$. We need to show that it is either of dictator type or of box type.

We can think of the input to $f_1,\dots,f_m$ as an $(n+1) \times m$ matrix whose rows belong to $P$. The definition of polymorphism states that if we apply $f_1,\dots,f_m$ on the columns of the matrix, then we obtain a member of $Q$.
If we fix the last row of the matrix to any member of $P$ then we obtain an $n$-ary polymorphism to which we can apply the assumption that $(P,Q)$ is $\Phi$-box-trivial for $n$.

Formally, define $f_i|_\sigma\colon \Sigma_i^n \to \Delta_i$ for $\sigma \in \Sigma_i$ by $f_i|_\sigma(x) = f_i(x,\sigma)$. 
If $(\sigma_1,\dots,\sigma_m) \in P$ then $(f_1|_{\sigma_1},\dots,f_m|_{\sigma_m})$ is an $n$-ary polymorphism of $(P,Q)$, and so by assumption, it is either of dictator type or of box type.

If $(f_1,\dots,f_m)$ is of box type then we are done. Otherwise, there exists $(\omega_1,\dots,\omega_m) \notin Q$ such that for all $i$ we have $\omega_i \in \im f_i$, say $\omega_i \in \im f_i|_{\alpha_i}$. We will use this ``witness'' to show that $(f_1,\dots,f_m)$ is of dictator type.

\paragraph{\boldmath Reduction to $n = 2$}
The main idea of the proof is to construct a $2$-ary polymorphism $(g_1,\dots,g_m)$ which ``encapsulates'' the original polymorphism $(f_1,\dots,f_m)$, and to use the assumption that $(P,Q)$ is $\Phi$-box-trivial for $n = 2$.
We define $g_i$ so that $g_i|_\sigma$ captures the behavior of $f_i|_\sigma$:
\begin{itemize}
\item If $f_i|_\sigma$ is the constant $c$ function, then we define $g_i|_\sigma \equiv c$.
\item If $f_i|_\sigma(x) = \phi(x_s)$ for some non-constant $\phi$, then we define $g_i|_\sigma = \phi$.
\item If $f_i|_\sigma$ depends on at least two coordinates, then we choose $g_i|_\sigma$ to be an arbitrary function satisfying
\begin{enumerate}[(i)]
\item $\im g_i|_\sigma \subseteq \im f_i|_\sigma$.
\item $g_i|_\sigma$ is non-constant.
\item If $\sigma = \alpha_i$ then $\omega_i \in \im g_i|_\sigma$.
\end{enumerate}
Let us explain why this is possible. Since $f_i|_\sigma$ depends on at least two coordinates, in particular it is non-constant, and so we can choose two different $a,b \in \im f_i|_\sigma$. Since $|\Sigma_i| \ge 2$, we can guarantee that $\im g_i|_\sigma = \{a,b\}$, and in particular $g_i|_\sigma$ is non-constant. Finally, if $\sigma = \alpha_i$ then by definition $\omega_i \in \im f_i|_\sigma$, and so we can arrange that $a,b$ include $\omega_i$.
\end{itemize}

Before showing that $(g_1,\dots,g_m)$ is indeed a polymorphism, we record two useful properties of the construction:
\begin{itemize}
\item $\omega_i \in \im g_i|_{\alpha_i}$ for all $i$.

Indeed, recall that $\omega_i \in \im f_i|_{\alpha_i}$. If $f_i|_{\alpha_i}$ depends on at least two coordinates then $\omega_i \in \im g_i|_{\alpha_i}$ by construction, and otherwise this holds since $\im g_i|_{\alpha_i} = \im f_i|_{\alpha_i}$.

\item $\im g_i|_\sigma \subseteq \im f_i|_\sigma$ for all $i$.

Indeed, if $f_i|_\sigma$ is constant or depends on a single coordinate then $\im g_i|_\sigma = \im f_i|_\sigma$, and otherwise $\im g_i|_\sigma \subseteq \im f_i|_\sigma$ by construction.
\end{itemize}

\paragraph{\boldmath $(g_1,\dots,g_m)$ is a polymorphism} In order to show that $(g_1,\dots,g_m)$ is a polymorphism, it suffices to show that $(g_1|_{\sigma_1},\dots,g_m|_{\sigma_m})$ is a polymorphism for all $(\sigma_1,\dots,\sigma_m) \in P$.

To this end, recall that $(f_1|_{\sigma_1},\dots,f_m|_{\sigma_m})$ is an $n$-ary polymorphism. By assumption, $(f_1|_{\sigma_1},\dots,f_m|_{\sigma_m})$ is either of dictator type or of box type.
If $(f_1|_{\sigma_1},\dots,f_m|_{\sigma_m})$ is of dictator type, say $f_i|_{\sigma_i} = \phi_i(x_s)$ for all $i$, then $g_i|_{\sigma_i} = \phi_i$ for all $i$, and so $(g_1|_{\sigma_1},\dots,g_m|_{\sigma_m})$ is a polymorphism.
If $(f_1|_{\sigma_1},\dots,f_m|_{\sigma_m})$ is of box type then since $\im g_i|_{\sigma_i} \subseteq \im f_i|_{\sigma_i}$ for all $i$, also $(g_1|_{\sigma_1},\dots,g_m|_{\sigma_m})$ is of box type, and so a polymorphism. 

\paragraph{\boldmath Applying $\Phi$-box-triviality for $n = 2$}
By assumption, $(g_1,\dots,g_m)$ is $\Phi$-box-trivial, and so it is either of dictator type or of box type.
Since $\omega_i \in \im g_i$ for all $i$ and $(\omega_1,\dots,\omega_m) \notin Q$, we see that $(g_1,\dots,g_m)$ cannot be of box type, and so must be of dictator type: there exist $s \in \{1,2\}$ and $(\phi_1,\dots,\phi_m) \in \Phi$ such that $g_i(x) = \phi_i(x_s)$ for all $i$.

The easy case is when $s = 2$. In this case, $g_i|_\sigma$ is the constant $\phi_i(\sigma)$ function for all $i,\sigma$. Going over the cases in the definition of $g_i|_\sigma$, this is only possible if $f_i|_\sigma \equiv \phi_i(\sigma)$. We conclude that $f_i(x) = \phi_i(x_{n+1})$ for all $i$, and so $(f_1,\dots,f_m)$ is of dictator type.

\paragraph{\boldmath $(g_1,\dots,g_m)$ is a dictator on the first coordinate}

The more challenging case is when $s = 1$. In this case, $g_i|_\sigma = \phi_i$ for all $i,\sigma$, where $(\phi_1,\dots,\phi_m) \in \Phi$. For given $i,\sigma$, this could happen since $f_i|_\sigma = \phi_i(x_{s(i,\sigma)})$, or when $f_i|_\sigma$ depends on at least two coordinates. We would like to rule out the latter case.

To this end, consider any $i,\sigma_i$. By the assumption that $P|_i = \Sigma_i$, we can complete $\sigma_i$ to some $(\sigma_1,\dots,\sigma_m) \in P$. Then $(f_1|_{\sigma_1},\dots,f_m|_{\sigma_m})$ is either of dictator type or of box type. In the former case, we are done, so we need to rule out the latter case. Indeed, recall that $\omega_i \in \im g_i|_{\alpha_i} = \im g_i|_{\sigma_i} \subseteq \im f_i|_{\sigma_i}$ for all $i$, ruling out the possibility that $(f_1|_{\sigma_1},\dots,f_m|_{\sigma_m})$ is of box type (since $(\omega_1,\dots,\omega_m) \notin Q$).

We conclude that $f_i|_\sigma = \phi_i(x_{s(i,\sigma)})$ for all $i,\sigma$. If we could show that all $s(i,\sigma)$ coincide, it would follow that $(f_1,\dots,f_m)$ is of dictator type. In fact, it suffices to show that $s(i,\sigma)$ coincide for all $i$ such that $\phi_i$ is non-constant. We prove that this is the case using a proof by contradiction: assuming that this condition doesn't hold, we construct a $2$-ary polymorphism of $(P,Q)$ which is not $\Phi$-box-trivial.

Let $V$ consist of all $i$ such that $\phi_i$ is non-constant, and let $\phi_i \equiv c_i$ for $i \notin V$. Our analysis above shows that if $(\sigma_1,\dots,\sigma_m) \in P$ then there exists $s$ such that $f_i|_{\sigma_i}(x) = \phi_i(x_s)$ for all $i$. When $i \in V$, this defines $s$ uniquely, and so $s(i,\sigma_i) = s(j,\sigma_j)$ whenever $i,j \in V$.

We can express the above condition in the terminology of graphs. Consider the graph whose vertices are $(i,\sigma)$ for $i \in V$ and $\sigma \in \Sigma_i$, and in which there is an edge connecting $(i,\sigma_i)$ and $(j,\sigma_j)$ for all $(\sigma_1,\dots,\sigma_m) \in P$. If $(i,\sigma)$ is connected to $(j,\tau)$ then $s(i,\sigma) = s(j,\tau)$. The assumption that $s(i,\sigma)$ for $i \in V$ do not all coincide implies that the graph has more than one connected component.

Let $\bigcup_{i \in V} \{i\} \times \Sigma_{i,0}$ be an arbitrary connected component, and let $\Sigma_{i,1} = \Sigma_i \setminus \Sigma_{i,0}$. Since every connected component includes vertices of the form $(i,\cdot)$ for all $i \in V$, we see that $\Sigma_{i,0},\Sigma_{i,1}$ are both non-empty for all $i \in V$.

The definition of the graph implies that for each $(\sigma_1,\dots,\sigma_m) \in P$, either $\sigma_i \in \Sigma_{i,0}$ for all $i \in V$, or $\sigma_i \in \Sigma_{i,1}$ for all $i \in V$. We use this to construct a $2$-ary polymorphism $(\chi_1,\dots,\chi_m)$ which is not $\Phi$-box-trivial, defined as follows:
\[
 \chi_i(\sigma,\tau) =
 \begin{cases}
     c_i & \text{if } i \notin V, \\
     \phi_i(\sigma) & \text{if } i \in V \text{ and } \sigma \in \Sigma_{i,0}, \\
     \phi_i(\tau) & \text{if } i \in V \text{ and } \sigma \in \Sigma_{i,1}.
 \end{cases}
\]

To see that $(\chi_1,\dots,\chi_m)$ is a polymorphism, observe that applying these functions to $(\sigma_1,\dots,\sigma_m),\allowbreak(\tau_1,\dots,\tau_m) \in P$ results in either $(\phi_1(\sigma_1),\dots,\phi_m(\sigma_m)) \in Q$ or $(\phi_1(\tau_1),\dots,\phi_m(\tau_m)) \in Q$.

Since $\omega_i \in \im g_i|_{\alpha_i} = \im \phi_i = \im \chi_i$, we see that $(\chi_1,\dots,\chi_m)$ is not of box type (since $(\omega_1,\dots,\omega_m) \notin Q$). We complete the proof by showing that $\chi_i$ depends on both coordinates whenever $i \in V$ (we can assume that $V \neq \emptyset$, since otherwise we can view $(g_1,\dots,g_m)$ as a dictator on the second coordinate).

Let $i \in V$. Choose $\sigma_0 \in \Sigma_{i,0}$, $\sigma_1 \in \Sigma_{i,1}$, and $\tau \in \Sigma_i$ such that $\phi_i(\tau) \neq \phi_i(\sigma_0)$. The following calculations show that $\chi_i$ depends on both coordinates:
\begin{align*}
\chi_i(\sigma_0,\tau) &= \phi_i(\sigma_0), &
\chi_i(\sigma_1,\tau) &=
\phi_i(\tau); \\
\chi_i(\sigma_1,\sigma_0) &= \phi_i(\sigma_0), &
\chi_i(\sigma_1,\tau) &= \phi_i(\tau).
\end{align*}

Thus $(\chi_1,\dots,\chi_m)$ is a $2$-ary polymorphism which is not $\Phi$-box-trivial, contrary to assumption. This contradiction shows that $s(i,\sigma)$ must coincide for all $i \in V$ and $\sigma \in \Sigma_i$, say all are equal to $s$. Therefore $f_i(x) = \phi_i(x_s)$ for all $i$, and so $(f_1,\dots,f_m)$ is of dictator type.
\end{proof}

\subsection{Certificate version}
\label{sec:certificate-version}

Comparing \Cref{thm:main-binary} (when $P = Q$) to the corresponding result \cite[\oldbinary]{Filmus26}, both the assumption and the conclusion are weaker, and so it seems that these theorems are incomparable. However, it turns out that a result in the style of \cite[\oldbinary]{Filmus26} follows from \Cref{thm:main-binary}.

Let us first extend the definition of certificate-triviality to the setting of relation pairs.

\begin{definition}[Certificate-triviality]
Let $\Phi \subseteq \Delta_1^{\Sigma_1} \times \cdots \times \Delta_m^{\Sigma_m}$. 

An $m$-tuple of functions $(f_1,\dots,f_m)$, where $f_i\colon \Sigma_i^n \to \Delta_i$, is $\Phi$-certificate-trivial if either of the following cases holds:
\begin{description}
    \item[Dictator type] There exist $s \in [n]$ and $(\phi_1,\dots,\phi_m) \in \Phi$ such that $f_i(x) = \phi_i(x_s)$ for all $i$.
    \item[Certificate type] There exist a subset $I \subseteq [m]$ and symbols $\delta_i \in \Delta_i$ for all $i \in I$ such that (i) $f_i \equiv \delta_i$ for all $i \in I$, (ii) $x \in Q$ whenever $x_i = \delta_i$ for all $i \in I$.
\end{description}

A relation pair $(P,Q)$, where $P \subseteq \Sigma_1 \times \cdots \times \Sigma_m$ and $Q \subseteq \Delta_1 \times \cdots \times \Delta_m$, is $\Phi$-certificate-trivial for $n$, where $n \ge 1$, if every $n$-ary polymorphism of $(P,Q)$ is $\Phi$-certificate-trivial.
\end{definition}

We can now formulate and prove a result which extends \cite[\oldbinary]{Filmus26} to the setting of relation pairs.

\begin{corollary} \label{cor:main-binary}
Let $P \subseteq \Sigma_1 \times \cdots \times \Sigma_m$, $Q \subseteq \Delta_1 \times \cdots \times \Delta_m$, and $\Phi \subseteq \Delta_1^{\Sigma_1} \times \cdots \times \Delta_m^{\Sigma_m}$.

Suppose that $|\Sigma_i| \ge 2$ for all $i$, and that $P$ has full projections: $P|_i = \Sigma_i$.

The relation pair $(P,Q)$ is $\Phi$-certificate-trivial for all $n$ iff it is $\Phi$-certificate-trivial for $n=2$.
\end{corollary}
\begin{proof}
Suppose that $(P,Q)$ is $\Phi$-certificate-trivial for $n=2$. We need to show that $(P,Q)$ is $\Phi$-certificate-trivial for all $n$.

We claim that $(P,Q)$ is $\Phi$-box-trivial for $n=2$, and so for all $n$ by \Cref{thm:main-binary}. Indeed, suppose that $(f_1,\dots,f_m)$ is a $2$-ary polymorphism of $(P,Q)$. By assumption, it is either of dictator type or of certificate type. In the latter case, let $I \subseteq [m]$ and $\delta_i \in \Delta_i$ for $i \in I$ be the corresponding certificate. Observe that any $x \in \im f_1 \times \cdots \times \im f_m$ satisfies $x_i = \delta_i$ for all $i \in I$, and so $x \in Q$. Thus $(f_1,\dots,f_m)$ is also of box type.

Now let $(f_1,\dots,f_m)$ be an $n$-ary polymorphism of $(P,Q)$. We need to show that it is $\Phi$-certificate-trivial. By \Cref{thm:main-binary}, $(f_1,\dots,f_m)$ is $\Phi$-box-trivial. If it is of dictator type, we are done, so assume it is of box type.
If all $f_i$ are constant then $(f_1,\dots,f_m)$ is also of certificate type, so suppose that at least one of these is non-constant.

We construct a $2$-ary polymorphism $(g_1,\dots,g_m)$ satisfying the following properties:
\begin{itemize}
\item If $f_i$ is constant then $g_i = f_i$.
\item If $f_i$ is non-constant then
\begin{enumerate}[(i)]
\item $\im g_i \subseteq \im f_i$.
\item $|\im g_i| \ge 2$.
\item $g_i$ depends on both coordinates.
\end{enumerate}
To construct $g_i$, choose any two distinct $a,b \in \im f_i$, and define
\[
 g_i(c,d) =
 \begin{cases}
     a & \text{if } c = d, \\
     b & \text{otherwise}.
 \end{cases}
\]
\end{itemize}

Since $\im g_i \subseteq \im f_i$ for all $i$, we see that $(g_1,\dots,g_m)$ is of box type and so a polymorphism. Since $(P,Q)$ is $\Phi$-certificate-trivial for $n=2$, necessarily $(g_1,\dots,g_m)$ is either of dictator type or of certificate type.

Since at least one $f_i$ is non-constant, $(g_1,\dots,g_m)$ cannot be of dictator type, hence must be of certificate type. By construction, $g_i$ is constant iff $f_i$ is constant, and in that case they have the same value, and so $(f_1,\dots,f_m)$ is also of certificate type.
\end{proof}

\section{Reduction to unary}
\label{sec:reduction-to-unary}

In this section we prove \Cref{thm:main-unary}, combining ideas from the proofs of \cite[\oldunary]{Filmus26} and \cite[\improvedunary]{Filmus26b}.

\mainunary*

In view of \Cref{lem:Phi}, we can assume without loss of generality that all members of $\Phi$ are polymorphisms. Another easy observation is that $m > 1$: otherwise $P|_1 = \Sigma_1$ forces $Q = \Sigma_1$, contrary to assumption.

We introduce a useful piece of notation: if $|\Sigma_i| = 2$ and $\sigma \in \Sigma_i$, then $\bar{\sigma}$ is the other element of $\Sigma_i$. If $|\Sigma_1| = \cdots = |\Sigma_m| = 2$ and $y \in \Sigma_1 \times \cdots \times \Sigma_m$, then $\bar{y} = (\bar{y}_1,\dots,\bar{y}_m)$.

We start by showing how each of the obstructions leads to a $2$-ary polymorphism which is not $\Phi$-box-trivial:

\begin{description}
\item[Unit witness obstruction]
Let $\chi_j(\sigma,\tau) = \sigma$ for $j \neq i$, and
\[
 \chi_i(\sigma,\tau) =
 \begin{cases}
     \bar{b} & \text{if } \sigma = \tau = \bar{b}, \\
     b & \text{otherwise}.
 \end{cases}
\]

This is a polymorphism since applying $(\chi_1,\dots,\chi_m)$ to $z,w \in P$ either yields $z$ or $z$ with the $i$th coordinate modified to $b$.

Since $\chi_i$ depends on both inputs, $(\chi_1,\dots,\chi_m)$ is not of dictator type. It is not of box type since $\im \chi_j = \Sigma_j$ for all $j$, and we assume that $Q \neq \Sigma_1 \times \cdots \times \Sigma_m$.

\item[Generalized upset obstruction]
We define
\[
 \chi_i(\sigma,\tau) =
 \begin{cases}
     y_i & \text{if } \sigma = \tau = y_i, \\
     \bar{y}_i & \text{otherwise}.     
 \end{cases}
\]

Let us first check that this is a polymorphism. Let $z,w \in P$. If $z = y$ then applying $(\chi_1,\dots,\chi_m)$ to $(z,w)$ yields $w \in P$, and otherwise the result is obtained from $z$ by changing the $i$th coordinate to $\bar{y}_i$ whenever $w_i = \bar{y}_i$.

All $\chi_i$ depend on both inputs and satisfy $\im \chi_i = \Sigma_i$, and so $(\chi_1,\dots,\chi_m)$ is neither of dictator type nor of box type.
\end{description}

We now move to the main part of the proof: assuming that $(P,Q)$ is $\Phi$-box-trivial for $n = 1$, we show that any $2$-ary polymorphism $(f_1,\dots,f_m)$  is $\Phi$-box-trivial unless one of the obstructions occurs. \Cref{thm:main-binary} will then complete the proof.

The starting point of the proof is the observation that if we restrict one of the ``rows'' of the $2 \times m$ input matrix to $(f_1,\dots,f_m)$ to a member of $P$, then we obtain a $1$-ary polymorphism, to which the box-triviality assumption applies.
Formally, define $f_i\restfirst{\sigma}, f_i\restsecond{\sigma}\colon \Sigma_i \to \Delta_i$ by
\[
 f_i\restfirst{\sigma}(\tau) = f_i(\sigma,\tau), \qquad
 f_i\restsecond{\sigma}(\tau) = f_i(\tau,\sigma).
\]
For every $(\sigma_1,\dots,\sigma_m) \in P$, both $(f_1\restfirst{\sigma_1},\dots,f_m\restfirst{\sigma_m})$ and $(f_1\restsecond{\sigma_1},\dots,f_m\restsecond{\sigma_m})$ are $1$-ary polymorphisms, and so each of them is either of dictator type or of box type; in the former case, the tuple is a member of $\Phi$.

Here is an outline of the rest of the proof:
\begin{itemize}
\item We show that if all restrictions considered above are of dictator type then there is a (degenerate) generalized upset obstruction.

This is the only part using ideas from the proof of \cite[\improvedunary]{Filmus26b}.

\item We show that if all restrictions to the first argument (or to the second argument) are of box type then either $(f_1,\dots,f_m)$ is $\Phi$-box-trivial or there is a unit witness obstruction.

\item It remains to consider the case in which (without loss of generality) some restriction to the first argument is of dictator type, and another is of box type.
We complete the proof by considering several auxiliary polymorphisms and appealing both to $\Phi$-box-triviality for $n = 1$ and to the result in the preceding step.

The auxiliary polymorphisms that we employ differ between the case in which some alphabet satisfies $|\Sigma_i| \ge 3$ and the case in which all alphabets satisfy $|\Sigma_i| = 2$.
\end{itemize}

We start by considering the case in which all restrictions $(f_1\restfirst{\sigma_1},\dots,f_m\restfirst{\sigma_m}),(f_1\restsecond{\sigma_1},\dots,f_m\restsecond{\sigma_m})$ are of dictator type.

\begin{lemma} \label{lem:all-dictator}
Suppose that $(f_1,\dots,f_m)$ is a $2$-ary polymorphism of $(P,Q)$ such that for all $(\sigma_1,\dots,\sigma_m) \in P$, both $(f_1\restfirst{\sigma_1},\dots,f_m\restfirst{\sigma_m}),(f_1\restsecond{\sigma_1},\dots,f_m\restsecond{\sigma_m})$ are of dictator type.

Then there is a generalized upset obstruction.
\end{lemma}
\begin{proof}
The proof considers two cases:
\begin{itemize}
\item There exist two different $z,w \in P$ which agree on some coordinate.

We show that this case leads to a contradiction.
\item Any two different $z,w \in P$ differ on all coordinates.

This is only possible if all alphabets have the same size, and after identifying them, $P = \{ (\sigma,\dots,\sigma) : \sigma \in \Sigma \}$. If $|\Sigma| \ge 3$ then this contradicts $\Phi$-box-triviality for $n=1$, and if $|\Sigma| = 2$ then there is a (degenerate) generalized upset obstruction.
\end{itemize}

\paragraph{\boldmath There exist two different $z,w \in P$ agreeing on some coordinate}
Let $i$ be a coordinate such that $z_i = w_i$, and let $j$ be a coordinate such that $z_j \neq w_j$.

By assumption, $(f_1\restfirst{z_1},\dots,f_m\restfirst{z_m}),(f_1\restfirst{w_1},\dots,f_m\restfirst{w_m}) \in \Phi$.
Since $f_i\restfirst{z_i} = f_i\restfirst{w_i}$, by synchronicity $f_j\restfirst{z_j} = f_j\restfirst{w_j}$, and in particular $f_j(z_j,z_j) = f_j(w_j,z_j)$.
However, this contradicts the assumption $(f_1\restsecond{z_1},\dots,f_m\restsecond{z_m}) \in \Phi$, which implies that $f_j\restsecond{z_j}$ is a permutation.

\paragraph{\boldmath Any two different $z,w \in P$ differ on all coordinates}
Since $P|_1 = \Sigma_1$, for every $\sigma \in \Sigma_1$ there exists $x(\sigma) \in P$ such that $x(\sigma)_1 = \sigma$. By assumption, this is the unique such member of $P$. For any $i > 1$, by assumption the mapping $\pi_i(\sigma) = x(\sigma)_i$ is injective, and so a bijection between $\Sigma_1$ and $\Sigma_i$, since $P|_i = \Sigma_i$; in particular, all $\Sigma_i$ have the same cardinality. Defining $\pi_1 = \id$, this shows that
\[
 P = \{(\pi_1(\sigma), \dots, \pi_m(\sigma)) : \sigma \in \Sigma_1\}.
\]

The proof now splits into two cases, depending on whether $\Sigma_1$ is Boolean:
\begin{itemize}
\item The case $|\Sigma_1| = 2$:
we show that there is a generalized upset obstruction.

In this case $P$ consists of exactly two members $y,\bar{y}$. We show that there is a generalized upset obstruction based at $y$.

Indeed, let $x \in Q$ be different from $y$, say $x_{i_0} \neq y_{i_0}$. Consider the $1$-ary polymorphism $(\eta_1,\dots,\eta_m)$ defined by
\[
 \eta_i(y_i) = \bar{y}_i, \quad \eta_i(\bar{y}_i) = x_i.
\]
This is a polymorphism since the output is either $\bar{y}$ or $x$.

By $\Phi$-box-triviality, $(\eta_1,\dots,\eta_m)$ is either of dictator type or of box type. Since $\eta_{i_0}$ is constant, it must be of box type. Since $\im \eta_i = \{x_i,\bar{y}_i\}$, we see that $Q$ contains all vectors obtained from $x$ by changing coordinates to their value in $\bar{y}$.

\item The case $|\Sigma_1| \ge 3$:
we show that this case leads to a contradiction.

In order to simplify notation, we identify $\Sigma_1,\dots,\Sigma_m$ with a single alphabet $\Sigma$ of the same cardinality in a way that makes
\[
 P = \{(\sigma,\dots,\sigma) : \sigma \in \Sigma\}.
\]

Let $\sigma_0,\sigma_1 \in \Sigma$ be two different symbols. Define a function $g\colon \Sigma \to \Sigma$ as follows: $g(\sigma_0) = \sigma_1$, and $g(\sigma) = \sigma$ if $\sigma \neq \sigma_0$. Clearly $(g,\dots,g)$ is a polymorphism, and so it is either of dictator type or of box type. Since $g$ is not a permutation, $(g,\dots,g)$ must be of box type. Since $\im g = \Sigma \setminus \{\sigma_0\}$, this means that $y \in Q$ whenever $y_i \neq \sigma_0$ for all $i$.

Let $\pi_1,\dots,\pi_m$ be arbitrary permutations of $\Sigma$ fixing $\sigma_0$. The observation above implies that $(\pi_1,\dots,\pi_m)$ is a polymorphism, and so either of dictator type or of box type. Since $\im \pi_i = \Sigma$ for all $i$ and $Q \neq \Sigma^m$, necessarily $(\pi_1,\dots,\pi_m)$ is of dictator type, i.e., $(\pi_1,\dots,\pi_m) \in \Phi$.

Since $|\Sigma| \ge 3$, we can find two different permutations $\pi,\pi'$ that fix $\sigma_0$. The observation above implies that $(\pi,\pi,\dots,\pi),(\pi',\pi,\dots,\pi) \in \Phi$, contradicting synchronicity (since $m \ge 2$). \qedhere
\end{itemize}
\end{proof}

The next step is to consider the case in which all restrictions $(f_1\restfirst{\sigma_1},\dots,f_m\restfirst{\sigma_m})$ are of box type. We will have occasion to apply the following lemma to polymorphisms other than $(f_1,\dots,f_m)$, as well as to restrictions of the second argument rather than the first argument.

\begin{lemma} \label{lem:all-box}
Suppose that $(f_1,\dots,f_m)$ is a $2$-ary polymorphism of $(P,Q)$ such that for all $(\sigma_1,\dots,\sigma_m) \in P$, the polymorphism $(f_1\restfirst{\sigma_1},\dots,f_m\restfirst{\sigma_m})$ is of box type.

Then either $(f_1,\dots,f_m)$ is $\Phi$-box-trivial, or there is a unit witness obstruction. 
\end{lemma}
\begin{proof}
If $(f_1,\dots,f_m)$ is of box type then we are done. Otherwise, there exists $(\omega_1,\dots,\omega_m) \notin Q$ such that $\omega_i \in \im f_i$ for all $i$, say $\omega_i \in \im f_i\restfirst{\alpha_i}$.

Suppose that the functions $g_1,\dots,g_m$, where $g_i\colon \Sigma_i \to \Sigma_i$, satisfy the following conditions for all $i$:
\begin{enumerate}[(i)]
\item $g_i(\alpha_i) = \omega_i$.
\item $g_i(\sigma) \in \im f_i\restfirst{\sigma}$ for all $\sigma \neq \alpha_i$.
\end{enumerate}
We claim that this implies that $(g_1,\dots,g_m) \in \Phi$.

To see this, let us show first that $(g_1,\dots,g_m)$ is a polymorphism (in fact, this holds under the weaker assumption $g_i(\alpha_i) \in \im f_i\restfirst{\alpha_i}$). Given $(\sigma_1,\dots,\sigma_m) \in P$, observe that $g_i(\sigma_i) \in \im f_i\restfirst{\sigma_i}$ for all $i$, and so $(g_1(\sigma_1),\dots,g_m(\sigma_m)) \in Q$ since $(f_1\restfirst{\sigma_1},\dots,f_m\restfirst{\sigma_m})$ is of box type. By $\Phi$-box-triviality for $n=1$, we see that either $(g_1,\dots,g_m) \in \Phi$ or $(g_1,\dots,g_m)$ is of box type; but the latter is impossible since $\omega_i \in \im g_i$ for all $i$, and $(\omega_1,\dots,\omega_m) \notin Q$.

The observation above implies that $f_i\restfirst{\sigma}$ is constant for all $i,\sigma \neq \alpha_i$. Indeed, suppose that this fails for some $i,\sigma \neq \alpha_i$. Then we can construct two different functions $g_i,g'_i$ satisfying the constraints above. Let $g_j$ for $j \neq i$ also satisfy the above constraints. Then $(g_1,\dots,g_i,\dots,g_m),(g_1,\dots,g'_i,\dots,g_m) \in \Phi$, contradicting synchronicity (since $m \ge 2$).

We complete the proof by considering two cases:
\begin{itemize}
\item $f_i\restfirst{\alpha_i}$ is constant for all $i$: we show that $f_1,\dots,f_m$ is of dictator type.

Indeed, in this case $f_1,\dots,f_m$ only depend on the first coordinate. Choosing $(\sigma_1,\dots,\sigma_m) \in P$ arbitrarily, $(f_1,\dots,f_m)$ essentially coincides with the polymorphism $(f_1\restsecond{\sigma_1},\dots,f_m\restsecond{\sigma_m})$. By $\Phi$-box-triviality for $n = 1$, the latter is either of dictator type or of box type. In the former case, $(f_1,\dots,f_m)$ is also of dictator type. The latter case cannot happen since $\omega_i \in \im f_i = \im f_i\restsecond{\sigma_i}$ for all $i$ and $(\omega_1,\dots,\omega_m) \notin Q$.

\item $f_i\restfirst{\alpha_i}$ is non-constant for some $i = i_0$: we show that there is a unit witness obstruction.

Recall the functions $(g_1,\dots,g_m)$ defined above; we have shown above that $(g_1,\dots,g_m) \in \Phi$. Let $g'_{i_0}$ be obtained from $g_{i_0}$ by modifying $g_{i_0}(\alpha_{i_0})$ to any value in $\im f_{i_0}\restfirst{\alpha_{i_0}} \setminus \{\omega_{i_0}\}$ (such a value exists by assumption). As shown above, $(g_1,\dots,g'_{i_0},\dots,g_m)$ is also a polymorphism, and so by $\Phi$-box-triviality for $n=1$, it is either of dictator type or of box type. 

Since $(g_1,\dots,g_m) \in \Phi$, it follows that $g_{i_0}$ is a permutation, and so $g'_{i_0}$ is not, hence $(g_1,\dots,g'_{i_0},\dots,g_m) \notin \Phi$. It follows that $(g_1,\dots,g'_{i_0},\dots,g_m)$ must be of box type. Since $\im g_i = \Sigma_i$ for $i \neq i_0$ while $\im g'_{i_0} = \Sigma_{i_0} \setminus \{\omega_{i_0}\}$, this shows that $y \in Q$ whenever $y_{i_0} \neq \omega_{i_0}$.

If $|\Sigma_{i_0}| = 2$ then $y \in Q$ whenever $y_{i_0} = \bar{\omega}_{i_0}$, and so there is a unit witness obstruction.

We complete the proof by ruling out the case $|\Sigma_{i_0}| \ge 3$. Let $\alpha,\beta \in \Sigma_{i_0} \setminus \{\omega_{i_0}\}$ be two different symbols, and let $h_{i_0}$ be the transposition that swaps $\alpha$ and $\beta$. It is easy to check that $(\id,\dots,h_{i_0},\dots,\id)$ is a polymorphism, which by $\Phi$-box-triviality for $n = 1$ must belong to $\Phi$ (since $Q \neq \Sigma_1 \times \cdots \times \Sigma_m$ rules out box type). Similarly, $(\id,\dots,\id) \in \Phi$. Since $m \ge 2$, this contradicts synchronicity. \qedhere
\end{itemize}
\end{proof}

The preceding two lemmas allow us to assume (up to switching the order of coordinates) that some restriction $(f_1\restfirst{\sigma_1},\dots,f_m\restfirst{\sigma_m})$ is of box type and another $(f_1\restfirst{\tau_1},\dots,f_m\restfirst{\tau_m})$ is of dictator type. These cases are mutually exclusive since all members of $\Phi$ consist of permutations, and $Q \neq \Sigma_1 \times \cdots \times \Sigma_m$.

The remainder of the proof depends on whether all alphabets are Boolean (have size $2$), or at least one is non-Boolean.
We start with the case in which some alphabet is non-Boolean.

\begin{lemma} \label{lem:some-non-boolean}
Suppose that $(f_1,\dots,f_m)$ is a $2$-ary polymorphism of $(P,Q)$ such that for some $(\sigma_1,\dots,\sigma_m),\allowbreak(\tau_1,\dots,\tau_m) \in P$, the polymorphism $(f_1\restfirst{\sigma_1},\dots,f_m\restfirst{\sigma_m})$ is of box type and $(f_1\restfirst{\tau_1},\dots,f_m\restfirst{\tau_m}) \in \Phi$.

Assume that $|\Sigma_{i_0}| \ge 3$ for some $i_0$.

Then either $(f_1,\dots,f_m)$ is of box type, or there is a unit witness obstruction.
\end{lemma}
\begin{proof}
Assuming that $(f_1,\dots,f_m)$ is not of box type and that there is no unit witness obstruction, we need to reach a contradiction.
The proof is composed of two parts:
\begin{enumerate}
\item We find $y \in P$ such that for all $z \in P$ different from $y$, the following holds: if $w$ is such that $z_i \neq y_i \to w_i \neq y_i$ for all $i$, then $w \in Q$.
\item We use this information together with the assumption $|\Sigma_{i_0}| \ge 3$ to construct a $1$-ary polymorphism which is not $\Phi$-box-trivial.
\end{enumerate}

We find $y$ by considering a certain $2$-ary polymorphism $(h_1,\dots,h_m)$, and we show that it satisfies the claimed closure property by considering a $1$-ary polymorphism $(\eta_1,\dots,\eta_m)$ whose definition depends on $z$.

\paragraph{\boldmath First part: finding $y$} 
The $2$-ary polymorphism $(h_1,\dots,h_m)$ is defined as follows:
\[
 h_i(a,b) =
 \begin{cases}
     b & \text{if } f_i\restfirst{a} \text{ is a permutation}, \\
     f_i(a,b) & \text{otherwise}.
 \end{cases}
\]

To check that this is indeed a polymorphism, it suffices to show that $(h_1\restfirst{x_1},\dots,h_m\restfirst{x_m})$ is a polymorphism for all $x \in P$. Since $(f_1\restfirst{x_1},\dots,f_m\restfirst{x_m})$ is a polymorphism, by $\Phi$-box-triviality for $n = 1$, $(f_1\restfirst{x_1},\dots,f_m\restfirst{x_m})$ is either of dictator type or of box type. In the first case, $(h_1\restfirst{x_1},\dots,h_m\restfirst{x_m}) = (\id,\dots,\id)$ is clearly a polymorphism. In the latter case, $\im h_i\restfirst{x_i} = \im f_i\restfirst{x_i}$ for all $i$, and so $(h_1\restfirst{x_1},\dots,h_m\restfirst{x_m})$ is also of box type and so a polymorphism.

\Cref{lem:all-box} (applied to the second coordinate) shows that either $(h_1\restsecond{y_1},\dots,h_m\restsecond{y_m})$ is of dictator type for some $y \in P$, or $(h_1,\dots,h_m)$ is $\Phi$-box-trivial (we assumed that there is no unit witness obstruction).
We would like to rule out the possibility that $(h_1,\dots,h_m)$ is $\Phi$-box-trivial.

Since $(f_1\restfirst{\tau_1},\dots,f_m\restfirst{\tau_m}) \in \Phi$, it follows that $h_i\restfirst{\tau_i} = \id$ for all $i$. This shows that $\im h_i = \Sigma_i$, ruling out the possibility that $(h_1,\dots,h_m)$ is of box type (since $Q \neq \Sigma_1 \times \cdots \times \Sigma_m$). It also shows that all $h_i$ depend on the second coordinate.

In order to complete the proof that $(h_1,\dots,h_m)$ is not $\Phi$-box-trivial, it remains to show that some $h_i$ depends on the first coordinate. We do this by considering $(f_1\restfirst{\sigma_1},\dots,f_m\restfirst{\sigma_m})$, which is of box type. Since $Q \neq \Sigma_1 \times \cdots \times \Sigma_m$, there must be $i$ such that $\im f_i\restfirst{\sigma_i} \neq \Sigma_i$, say $\kappa \notin \im f_i\restfirst{\sigma_i}$. In particular, $f_i\restfirst{\sigma_i}$ is not a permutation.
Then $h_i(\tau_i,\kappa) = \kappa$ whereas $h_i(\sigma_i,\kappa) = f_i(\sigma_i,\kappa) \neq \kappa$, showing that $h_i$ depends on the first coordinate.

Summarizing, we have shown that $(h_1\restsecond{y_1},\dots,h_m\restsecond{y_m}) \in \Phi$ for some $y \in P$.
This has the following implication: if $h_i\restfirst{a} = \id$ then $a = \tau_i$. Indeed, $h_i\restfirst{\tau_i} = \id$ since $(h_1\restfirst{\tau_1},\dots,h_m\restfirst{\tau_m}) \in \Phi$, and so $h_i(a,y_i) = h_i(\tau_i,y_i) = y_i$. Since $h_i\restsecond{y_i}$ is a permutation, this implies that $a = \tau_i$.

When proving that $(h_1,\dots,h_m)$ is a polymorphism, we showed that for every $x \in P$, either $h_i\restfirst{x_i} = \id$ for all $i$, in which case $x_i = \tau_i$ for all $i$, or $(h_1\restfirst{x_1},\dots,h_m\restfirst{x_m})$ is of box type.
The upshot is that if $x \in P$ is different from $(\tau_1,\dots,\tau_m)$ then $(h_1\restfirst{x_1},\dots,h_m\restfirst{x_m})$ is of box type.

\paragraph{First part: proving the closure property}
Let $z \in P$ be different from $y$. We need to show that $w \in Q$ whenever $z_i \neq y_i \to w_i \neq y_i$ for all $i$. 

To this end, we consider the $1$-ary polymorphism $(\eta_1,\dots,\eta_m)$ defined as follows:
\[
 \eta_i(a) =
 \begin{cases}
     z_i & \text{if } a = \tau_i, \\
     h_i(a, y_i) & \text{otherwise}.
 \end{cases}
\]
This is a polymorphism since $(\eta_1(\tau_1),\dots,\eta_m(\tau_m)) = z \in P$, and if $x \neq (\tau_1,\dots,\tau_m)$, then $(h_1\restfirst{x_1},\dots,h_m\restfirst{x_m})$ is of box type.

By $\Phi$-box-triviality for $n = 1$, we see that either $(\eta_1,\dots,\eta_m) \in \Phi$ or $(\eta_1,\dots,\eta_m)$ is of box type. We can rule out the former case by showing that $\eta_i$ is not a permutation for any coordinate $i$ such that $z_i \neq y_i$. Indeed, recall that $h_i\restsecond{y_i}$ is a permutation. We obtain $\eta_i$ by changing its value at $\tau_i$ to $z_i$. Since $h_i(\tau_i,y_i) = y_i \neq z_i$, we deduce that $\eta_i$ is indeed not a permutation.

We conclude that $(\eta_1,\dots,\eta_m)$ is of box type. As we have just mentioned, $\eta_i$ is obtained from the permutation $h_i\restsecond{y_i}$ by changing the value at $\tau_i$ from $y_i$ to $z_i$. Hence $\im \eta_i = \Sigma_i$ if $z_i = y_i$, and $\im \eta_i = \Sigma_i \setminus \{y_i\}$ otherwise. The closure property now follows from the definition of box type.

\paragraph{Second part}
We complete the proof by considering the $1$-ary polymorphism $(\chi_1,\dots,\chi_m)$ defined as follows:
\[
 \chi_i(a) =
 \begin{cases}
     y_i & \text{if } a = y_i, \\
     \omega_i & \text{if } a \neq y_i \text{ and } y_i \neq \omega_i, \\
     \kappa_i & \text{if } a \neq y_i \text{ and } y_i = \omega_i, \text{for an arbitrary } \kappa_i \neq \omega_i.
 \end{cases}
\]
This is a polymorphism since $(\chi_1(y_1),\dots,\chi_m(y_m)) = y \in P$, and if $z \neq y$ then $(\chi_1(z_1),\dots,\chi_m(z_m))$ is obtained from $(z_1,\dots,z_m)$ by changing coordinates $z_i$ such that $z_i \neq y_i$ to some value different from $y_i$.

We complete the proof by contradiction by showing that $(\chi_1,\dots,\chi_m)$ is not $\Phi$-box-trivial, contradicting $\Phi$-box-triviality for $n = 1$. It is not of dictator type since $\chi_{i_0}$ is not a permutation (recall that $|\Sigma_{i_0}| \ge 3$). It is not of box type since $\omega_i \in \im \chi_i$ for all $i$, and $(\omega_1,\dots,\omega_m) \notin Q$.
\end{proof}

We complete the proof by considering the case in which all alphabets are Boolean.

\begin{lemma} \label{lem:all-boolean}
Suppose that $(f_1,\dots,f_m)$ is a $2$-ary polymorphism of $(P,Q)$ such that for some $(\sigma_1,\dots,\sigma_m),\allowbreak(\tau_1,\dots,\tau_m) \in P$, the polymorphism $(f_1\restfirst{\sigma_1},\dots,f_m\restfirst{\sigma_m})$ is of box type and $(f_1\restfirst{\tau_1},\dots,f_m\restfirst{\tau_m}) \in \Phi$.

Assume that $|\Sigma_i| = 2$ for all $i$.

Then either $(f_1,\dots,f_m)$ is $\Phi$-box-trivial, or there is a unit witness obstruction, or there is a generalized upset obstruction.
\end{lemma}
\begin{proof}
The overall plan is similar to that of the proof of \Cref{lem:some-non-boolean}, though the details are different. Assuming that $(f_1,\dots,f_m)$ is not of box type and that there is no unit witness obstruction, we show that there is a generalized upset obstruction.

We find the ``direction'' $y \in P$ in the definition of the generalized upset obstruction by considering an appropriate $2$-ary polymorphism and appealing to \Cref{lem:all-box}; this polymorphism is different from the one considered in the proof of \Cref{lem:some-non-boolean}. We then prove the generalized upset property.

\paragraph{\boldmath Finding $y$}
Consider the $2$-ary polymorphism $(h_1,\dots,h_m)$ defined as follows:
\[
 h_i(a,b) = f_i(a,f_i(a,b)).
\]

To check that this is indeed a polymorphism, consider any $x,t \in P$. Since $(f_1\restfirst{x_1},\dots,f_m\restfirst{x_m})$ is a polymorphism, by $\Phi$-box-triviality for $n = 1$, $(f_1\restfirst{x_1},\dots,f_m\restfirst{x_m})$ is either of dictator type or of box type. In the former case, $f_i\restfirst{x_i}$ is a permutation for all $i$, hence $f_i(x_i,f_i(x_i,t_i)) = t_i$ (since $f_i\restfirst{x_i}$ is either the identity or the transposition swapping the two elements of $\Sigma_i$). This shows that $(h_1(x_1,t_1),\dots,h_m(x_m,t_m)) \in P$. In the latter case, $(h_1(x_1,t_1),\dots,h_m(x_m,t_m)) \in Q$ since $h_i(x_i,t_i) \in \im f_i\restfirst{x_i}$ for all $i$.

\Cref{lem:all-box} (applied to the second coordinate) shows that either $(h_1\restsecond{y_1},\dots,h_m\restsecond{y_m})$ is of dictator type for some $y \in P$, or $(h_1,\dots,h_m)$ is $\Phi$-box-trivial (we assumed that there is no unit witness obstruction).
We would like to rule out the possibility that $(h_1,\dots,h_m)$ is $\Phi$-box-trivial.

First, we observe that $\im h_i\restfirst{\sigma} = \im f_i\restfirst{\sigma}$ for all $i,\sigma$. Indeed, if $f_i\restfirst{\sigma}$ is constant then $h_i\restfirst{\sigma}$ is equal to the same constant, and otherwise $\im h_i \restfirst{\sigma} = \im f_i \restfirst{\sigma} = \Sigma_i$. Thus $\im h_i = \im f_i$ for all $i$, and so if $(h_1,\dots,h_m)$ were of box type then so would $(f_1,\dots,f_m)$ be, contrary to assumption.

Second, $f_i\restfirst{\tau_i}$ is a permutation for all $i$, and so $h_i(\tau_i,b) = b$, showing that all $h_i$ depend on the second coordinate.
Since $(f_1\restfirst{\sigma_1},\dots,f_m\restfirst{\sigma_m})$ is of box type and $Q \neq \Sigma_1 \times \cdots \times \Sigma_m$, there must exist $i$ such that $\im f_i\restfirst{\sigma_i} \neq \Sigma_i$, and so $f_i\restfirst{\sigma_i}$ is equal to some constant $c_i$. Since $h_i(\tau_i,\bar{c}_i) = \bar{c}_i$ while $h_i(\sigma_i,\bar{c}_i) = c_i$, we see that $h_i$ depends on both coordinates, and so $(h_1,\dots,h_m)$ cannot be of dictator type.

We conclude that $(h_1\restsecond{y_1},\dots,h_m\restsecond{y_m})$ is of dictator type for some $y \in P$.
As in the proof of \Cref{lem:some-non-boolean}, this implies that if $h_i\restfirst{a}$ is a permutation then $a = \tau_i$. Indeed, by construction, if $h_i\restfirst{a}$ is a permutation then it must be the identity. Since $h_i(a,y_i) = h_i(\tau_i,y_i) = y_i$ and $h_i\restsecond{y_i}$ is a permutation, we see that $a = \tau_i$. Thus $h_i\restfirst{\bar{\tau}_i}$ is constant.

Since $h_i(\tau_i,y_i) = y_i$ and $h_i\restsecond{y_i}$ is a permutation, we see that $h_i(\bar{\tau}_i,y_i) = \bar{y}_i$. This shows that
\[
 h_i(a,b) =
 \begin{cases}
     b & \text{if } a = \tau_i, \\
     \bar{y}_i & \text{if } a = \bar{\tau}_i.
 \end{cases}
\]

Suppose that $z \in P$, where $z \neq (\tau_1,\dots,\tau_m)$. Since $(h_1,\dots,h_m)$ is a polymorphism, so is $(h_1\restfirst{z_1},\dots,h_m\restfirst{z_m})$. By $\Phi$-box-triviality for $n = 1$, it is either of dictator type or of box type. It cannot be of dictator type since one of its coordinates is constant, hence it must be of box type.
We conclude that if $z \in P$ satisfies $z \neq (\tau_1,\dots,\tau_m)$ and $w$ is such that $z_i = \bar{\tau}_i \to w_i = \bar{y}_i$ for all $i$, then $w \in Q$.

\paragraph{Proving the generalized upset property}
Let $x \in Q$ be different from $y$, say $x_{i_0} \neq y_{i_0}$. We need to show that $Q$ contains all vectors obtained from $x$ by changing coordinates to their value in $\bar{y}$.

We do this by considering the polymorphism $(\eta_1,\dots,\eta_m)$ defined by
\[
 \eta_i(\tau_i) = x_i, \quad \eta_i(\bar{\tau}_i) = \bar{y}_i.
\]
This is a polymorphism since $(\eta_1(\tau_1),\dots,\eta_m(\tau_m)) = x \in Q$, and if $z \in P$ is such that $z \neq (\tau_1,\dots,\tau_m)$ then $w = (\eta_1(z_1),\dots,\eta_m(z_m))$ satisfies the property $z_i = \bar{\tau}_i \to w_i = \bar{y}_i$ for all $i$, and so $w \in Q$.

By $\Phi$-box-triviality for $n = 1$, we see that $(\eta_1,\dots,\eta_m)$ is either of dictator type or of box type. Since $x_{i_0} = \bar{y}_{i_0}$, the function $\eta_{i_0}$ is constant, and so $(\eta_1,\dots,\eta_m)$ must be of box type. This implies the generalized upset property since $\im \eta_i = \{x_i, \bar{y}_i\}$ for all $i$.
\end{proof}

\bibliographystyle{alpha}
\bibliography{biblio}

\end{document}